\documentclass[11pt]{article}

\usepackage[latin1]{inputenc}
\usepackage{t1enc}
\usepackage{amsmath, amssymb, amsfonts, amsthm, amsopn,epsfig}
\usepackage[none]{hyphenat}

\usepackage{tikz}
\usetikzlibrary{math,calc,intersections}
\usetikzlibrary{positioning,arrows,shapes,decorations.markings,decorations.pathreplacing,matrix,patterns}
\tikzstyle{vertex}=[circle,draw=black,fill=black,inner sep=0,minimum size=5pt,text=white,font=\footnotesize]

\usepackage{float}
\usepackage{graphicx}
\usepackage{caption}
\usepackage[margin=1in]{geometry}
\usepackage[toc,page]{appendix}
\usepackage{epstopdf}
\usepackage{etoolbox}
\usepackage[colorlinks=true]{hyperref}
\usepackage{dsfont}
\hypersetup{
	colorlinks=true,
	linkcolor=blue,
	filecolor=magenta,      
	urlcolor=cyan,
	citecolor=blue
}
\usepackage{cleveref}

\theoremstyle{plain}
\newtheorem{theorem}{Theorem}[section]

\newtheorem{corollary}[theorem]{Corollary}
\newtheorem{claim}[theorem]{Claim}
\newtheorem{lemma}[theorem]{Lemma}

\theoremstyle{definition}

\title{\vspace{-0.8cm}Evasive sets, covering by subspaces, and point-hyperplane incidences}
\author{Benny Sudakov  \thanks{ETH Zurich, \emph{e-mail}: \textbf{\{benjamin.sudakov,istvan.tomon\}@math.ethz.ch}}\and Istv\'an Tomon\footnotemark[1] }
\date{}

\begin{document}
	\sloppy 
	\maketitle
	
	\begin{abstract}
		Given positive integers $k\leq d$ and a finite field $\mathbb{F}$, a set $S\subset\mathbb{F}^{d}$ is \emph{$(k,c)$-subspace evasive} if every $k$-dimensional affine subspace  contains at most $c$ elements of $S$.  
		By a simple averaging argument, the maximum size of a $(k,c)$-subspace evasive set is at most $c |\mathbb{F}|^{d-k}$. When $k$ and $d$ are fixed, and $c$ is sufficiently large, the matching lower bound $\Omega(|\mathbb{F}|^{d-k})$ is proved by Dvir and Lovett. We provide an alternative proof of this result using the random algebraic method. We also prove sharp upper bounds on the size of $(k,c)$-evasive sets in case $d$ is large, extending results of Ben-Aroya and Shinkar. 
		
		The existence of optimal evasive sets has several interesting consequences in combinatorial geometry. We show that the minimum number of $k$-dimensional linear hyperplanes needed to cover the grid $[n]^{d}\subset \mathbb{R}^{d}$ is $\Omega_{d}\big(n^{\frac{d(d-k)}{d-1}}\big)$, which matches the upper bound proved by Balko, Cibulka, and Valtr, and settles a problem proposed by Brass, Moser, and Pach. Furthermore, we improve the best known lower bound on the maximum number of incidences between points and hyperplanes in $\mathbb{R}^{d}$ assuming their incidence graph avoids the complete bipartite graph $K_{c,c}$ for some large constant $c=c(d)$.
	\end{abstract}
	
	\section{Introduction}
	
	Given a finite field $\mathbb{F}$, a set of points $S\subset \mathbb{F}^{d}$ is \emph{$(k,c)$-subspace evasive} if no $k$-dimensional affine subspace contains more than $c$ elements of $S$. This notion was first investigated in an influential work of Pudl\'ak and R\"odl \cite{PR04}, who observed that explicit constructions of evasive sets in $\mathbb{F}_2^{n}$ can be transformed to give explicit constructions of bipartite Ramsey graphs. In particular, they showed that a $(d/2,c)$-evasive set $S\subset \mathbb{F}_2^{d}$ can be used to construct a bipartite graph with vertex classes of size $|S|$ containing no complete or empty bipartite graph with parts of size more than $c$. Evasive sets  also have application in coding theory in the context of list-decoding, and in combinatorial geometry, where they can be used to get incidence bounds.
	
	\subsection{Evasive sets and coding theory}
	
 Error-correcting codes are used for controlling errors in data transmission over noisy or unreliable communication channels and they were extensively studied in the last 70 years in information theory, computer science and telecommunication. An \emph{$[m,r,t]$-code over the field $\mathbb{F}$} is a linear subspace $L< \mathbb{F}^{m}$ of dimension $r$ such that the Hamming distance between any two distinct elements of $L$ is at least $t$, or equivalently, $L$ contains no nonzero vector with less than $t$ nonzero coordinates. In practice, an $[m,r,t]$-code can be used to send $r$ ($\mathbb{F}$-ary) bits of data using $m$ bits, and is capable of correcting $(t-1)/2$ faulty bits. In other words, the Hamming balls of radius $\lfloor(t-1)/2\rfloor$ centered at the code words of $L$ are disjoint, which gives the celebrated Hamming bound  $m\leq O_t\Big(|\mathbb{F}|^{\frac{(m-r)}{\lfloor(t-1)/2\rfloor}-1}\Big)$ (see, e.g., \cite{vanL}).
	
	A matrix $M\in\mathbb{F}^{(m-r)\times m}$, whose kernel is $L$, is a \emph{parity-check matrix} of $L$. It is easy to show that $L$ is an $[m,r,t]$-code if and only if any $t-1$ columns of $M$ are linearly dependent. Therefore, the problem of constructing $[m,r,t]$-codes is equivalent to the construction of a set $S$ of $m$ vectors in $\mathbb{F}^{m-r}$, forming the columns of $M$, such that no $(t-2)$-dimensional subspace contains $t-1$ elements of $S$. 
	Writing $d=m-r$ and $k=t-2$, the Hamming bound shows that if $S\subset \mathbb{F}^{d}$ such that no $k$-dimensional linear subspace contains $k+1$ elements of $S$, then \begin{equation}\label{eq:Hamming}
	    |S|\leq O_k\Big(|\mathbb{F}|^{\frac{d}{\lfloor(k+1)/2\rfloor}-1}\Big).
	\end{equation}
	
	\medskip
	
	A list-decoding problem deals with the case when we receive message with more than $(t-1)/2$ faulty bits. In this case we might not be able to uniquely determine the original message, but we can sometimes output a small list of possibilities. An error correcting code $L\subset \mathbb{F}^{m}$ is \emph{$(\rho,c)$ list-decodable} if the Hamming ball of radius $\rho m$ around every element of $L$  contains at most $c$ elements of $L$. In 2011, Guruswami \cite{G11} discovered an important  connection between evasive sets and \emph{list-decodable codes}. He showed in \cite{G11} that if $|\mathbb{F}|=d^{O(1/\varepsilon^2)}$ and there exists $S\subset \mathbb{F}^{d}$ such that $S$ is $(1/\varepsilon,c)$-subspace evasive of size $|S|\geq |\mathbb{F}|^{d(1-\epsilon)}$, then it is possible to construct a code $L\subset \mathbb{F}^{m}$ of size $|\mathbb{F}|^{\delta m}$ which is $(1-\delta-2\varepsilon,c)$ list-decodable.
	
	Furthermore, Guruswami \cite{G11} observed that a random set $S$ of size $|\mathbb{F}|^{d-k-\delta}$ is $\big(k,O(kd/\delta)\big)$-subspace evasive with high probability, and so such a set can be used to construct list-decodable codes of near optimal capacity. In this setting, one thinks of $k$ being fixed, while $d$ and (possibly $|\mathbb{F}|$) are large. Taking $\delta=\varepsilon d$ in the above result implies that a random set of $|\mathbb{F}|^{d(1-\varepsilon)}$ points is $(k,c)$-subspace evasive with $c=O(k/\varepsilon)$. Notably, $c$ does not depend on $|\mathbb{F}|$ or $d$. This simple probabilistic argument vastly outperforms every known explicit construction, so the main focus here is to find deterministic $(k,c)$-subspace evasive sets $S$ of size $|\mathbb{F}|^{d(1-\varepsilon)}$ with $c$ small as possible, see e.g. \cite{BS14, DL12}. On the other hand, Ben-Aroya and Shinkar \cite{BS14} proved that when $\varepsilon^{-1}\leq k^{O(1)}$, the bound $c=O(k/\varepsilon)$ cannot be improved, therefore the probabilistic construction is optimal. We extend this result, showing that it remains true for every $\varepsilon^{-1}=2^{O(k)}$ as well.
	
	\begin{theorem}\label{thm:lowerbound}
		Let $\mathbb{F}$ be a field, $k$ be a positive integer, and $0<\varepsilon<1/20$, then if $d$ is sufficiently large with respect to $k$, the following holds. Let $S\subset \mathbb{F}^{d}$ such that $|S|\geq |\mathbb{F}|^{d(1-\varepsilon)}$. Then $S$ is not $\big(k,\frac{k-\log_2(1/\varepsilon)}{8\varepsilon}\big)$-subspace evasive.
	\end{theorem}
	
	Theorem \ref{thm:lowerbound} shows that if $\varepsilon^{-1}=2^{k-1}$, then a set of size more than $|\mathbb{F}|^{d(1-\epsilon)}$ is not $(k,\Omega(2^{k}))$-evasive, while the result becomes meaningless if $\varepsilon^{-1}\geq 2^{k}$. Observe that the later is a natural barrier, as in case $\mathbb{F}=\mathbb{F}_2$, a $k$-dimensional affine subspace cannot contain more than $2^{k}$ points. Therefore, if one wants to extend Theorem \ref{thm:lowerbound} beyond $\varepsilon^{-1}\geq 2^{k}$, the field $\mathbb{F}$ also has to play some role. This setting, already for $k=1$, seems extremely difficult. Bounding the size of a set in $\mathbb{F}_3^d$ containing no three points on a line is equivalent with the famous \emph{Cap set problem}, for which the upper bound $2.756^d$ was recently proved by Ellenberg and Gijswijt \cite{EG17}, following the breakthrough of Croot, Lev, and Pach \cite{CLP}. However, similar results are already not known for four points on a line. 
	
	It appears, that the upper bound on evasiveness behaves very differently in the regime when $c$ is close to $k$. In this case, we show that the Hamming bound mentioned above can be used to estimate the size of  $(k,k+C)$-subspace evasive sets, where $C< k/2$. Interestingly, the method of proof for this 
	range of parameters is fundamentally different from that in Theorem \ref{thm:lowerbound}. While the proof of this theorem is mostly combinatorial, relying on a generalization of the Erd\H{o}s Box theorem \cite{E64}, the proof of the following result (such as the proof of the Hamming bound) is based on coding theory.
	
	\begin{theorem}\label{thm:Hammingbound}
		Let $S\subset \mathbb{F}^{d}$ be $(k,k+C)$-subspace evasive, where $C\leq \frac{k}{2}-1$. Then $$|S|\leq 4k|\mathbb{F}|^{\frac{d}{\lfloor k/2(C+1)\rfloor}}.$$
	\end{theorem}

\noindent	
Using the standard probabilistic argument, one can easily show that the bound in this theorem is optimal up to a factor of 2 in the exponent.
	
	\subsection{Evasive sets over large fields}
	
	Motivated by applications in combinatorial geometry, another interesting setting is to consider large $(k,c)$-subspace evasive sets in $\mathbb{F}^{d}$, where we think of $k$ and $d$ as fixed, while $|\mathbb{F}|$ is arbitrarily large. Clearly, a simple averaging argument shows that a $(k,c)$-subspace evasive set in $\mathbb{F}^{d}$ can have size at most $c|\mathbb{F}|^{d-k}$. As mentioned above, the probabilistic argument shows that a random set of $|\mathbb{F}|^{d-k-\delta}$ points is $(k,c)$-subspace evasive for $\delta=\Theta(kd/c)$. Note that, however, a random set of $\Omega_{d}(|\mathbb{F}|^{d-k})$ points does intersect many $k$-dimensional affine subspaces in  $\Omega_{d}(\log |\mathbb{F}|)$ elements with high probability. Dvir and Lovett \cite{DL12} (see Theorem 2.4 together with Claim 3.5) showed that this can be improved, by giving an explicit algebraic construction of a $(k,c)$-subspace evasive set of size $\Omega(|\mathbb{F}|^{d-k})$, where $c$ depends only on $d$ and~$k$.
	
	\begin{theorem}(\cite{DL12})\label{thm:main}
		For every pair of positive integers $k,d$ satisfying $k\leq d$, there exists a positive integer $c=c(d,k)$ such that the following holds. For every finite field $\mathbb{F}$, there exists a $(k,c)$-subspace evasive set of size $|\mathbb{F}|^{d-k}/3$ in $\mathbb{F}^{d}$.
	\end{theorem}
	
To provide a different perspective and for the convenience of the reader, we give a short, alternative proof of this theorem. Note the striking difference between the bound of Theorem \ref{thm:main} and the lower bounds in case $c$ being small. This leads to the natural question about the dependence of $c(d,k)$ on the parameters $d$ and $k$. The proof of Dvir and Lovett \cite{DL12} gives $c(d,k)=d^{k}$ (if $|\mathbb{F}|$ is sufficiently large), which is likely to be far from optimal, while our proof gives even worse bounds. On the other hand, applying Theorem \ref{thm:lowerbound} with $\varepsilon=\max\{\frac{k}{d},\frac{1}{2^{k-1}}\}$, we get the lower bound $c(d,k)=\Omega(\min\{d,2^k\})$. This might raise the question whether $c(d,k)$ can be bounded by a function of $k$ alone. However, this is not true already for $k=1$. Indeed, if $d$ is sufficiently large with respect to $k$ and $\mathbb{F}$, then the density Hales-Jewett theorem \cite{FK91} implies that any subset $S\subset \mathbb{F}^{d}$ of size at least $\frac{1}{|\mathbb{F}|^k}|\mathbb{F}|^d$ contains a \emph{combinatorial line}, which in turn is also a complete $1$-dimensional affine subspace.

	\subsection{Covering by subspaces}
	
	Theorem \ref{thm:main} has a number of interesting applications in combinatorial geometry. The following problem first appeared in a paper of  Brass and Knauer \cite{BK03} in connection to point-hyperplane incidences, which we discuss in more detail in the next subsection. Given positive integers $n,k,d,c$ with $k\leq d$, determine the maximum number of lattice points in the grid $[n]^{d}=\{1,\dots,n\}^{d}$ with no $k$-dimensional linear or affine subspace containing more than $c$ of them (over $\mathbb{R}$). Let $\ell(d,k,n,c)$ denote this maximum in the linear case, and $a(d,k,n,c)$ in the affine case.  Here, we are interested in the behavior of $\ell(d,k,n,c)$ and $a(d,k,n,c)$ as a function of $n$, while we think of $k,d,c$ as fixed. Clearly, we have $a(d,k,n,c)\leq cn^{d-k}$ as we can cover $[n]^{d}$ by $n^{d-k}$ affine hyperplanes of dimension $k$. On the other hand, a probabilistic argument of Brass and Knauer \cite{BK03} shows that for every $\varepsilon>0$ there exists $c=c(d,k,\varepsilon)$ such that  $a(d,k,n,c)\geq n^{d-k-\varepsilon}$. The tight result $a(d,k,n,k+1)=\Omega_{d}(n^{d-k})$ was only known in the
	two special cases when $k=1$ or $k=d-1$. A straightforward application of Theorem \ref{thm:main} lets us close the gap between the lower and upper bound for every $k<d$ and sufficiently large $c$.
	
	\begin{theorem}\label{thm:affine}
		For every pair of positive integers $k,d$ satisfying $k\leq d$, there exists a positive integer $c=c(d,k)$ such that the following holds. For every positive integer $n$ there exists a set $S\subset [n]^{d}$ of size at least $(n/2)^{d-k}$ such that no $k$-dimensional affine hyperplane contains more than $c$ elements of $S$.
	\end{theorem}
	
	Indeed, let $p$ be any prime between $n/2$ and $n$, which exists by Bertrand's postulate. Let $c=c(d,k)$ be the constant guaranteed by Theorem \ref{thm:main}, and let $S_0\subset \mathbb{F}_p^{d}$ be a set of $p^{d-k}\geq (n/2)^{d-k}$ vectors such that no $k$-dimensional affine subspace contains more than $c$ elements of $S_0$. Setting $S$ to be the set of lattice points in $[p]^{d}$ that are congruent to the elements of $S_0$ modulo $p$ gives the desired set.
	
	Determining $\ell(d,k,n,c)$ seems to be more difficult. Brass and Knauer \cite{BK03} conjectured that $\ell(d,k,n,k)=\Theta_{d}(n^{d(d-k)/(d-1)})$. However, this was refuted by Lefmann \cite{L} for most values of $k$ and $d$, as he showed that $\ell(d,k,n,k)=O_d(n^{d/\lfloor k/2\rfloor})$ (akin the Hamming bound, mentioned in the previous subsection).  Similarly to the affine case, bounding $\ell(d,k,n,c)$ is closely related to the problem of bounding $g(d,k,n)$, which is the minimum number of $k$-dimensional linear hyperplanes in a covering of $[n]^{d}$. Indeed, we trivially have $\ell(d,k,n,c)\leq cg(d,k,n)$. The problem of estimating $g(d,k,n)$ was proposed by Brass, Moser, and Pach \cite{BMP} (Problem 6 in Chapter 10.2). B\'ar\'any, Harcos, Pach, and Tardos \cite{BHPT} resolved the $k=d-1$ case of both problems by showing that $\Omega_d(n^{d/(d-1)})=\ell(d,d-1,n,d-1)\leq (d-1)g(d,d-1,n)=O_d(n^{d/(d-1)})$. In general, Balko, Cibulka, and Valtr \cite{BCV} showed that $g(d,k,n)=O_{d}(n^{d(d-k)/(d-1)})$ and $g(d,k,n)>n^{d(d-k)/(d-1)-o(1)}$, where the lower bound comes from proving $\ell(d,k,n,c)\geq n^{d(d-k)/(d-1)-\varepsilon}$ for some $\varepsilon=\varepsilon_{d,k}(c)$ tending to 0 as $c$ tends to infinity. If $k=1$, it was shown by Konyagin and Sudakov \cite{KS} that the $o(1)$ and $\varepsilon$ terms can be removed, closing the gap in this case. Here, we close the gap for all values of $k$ and~$d$.
	
	\begin{theorem}\label{thm:linear}
		For every pair of positive integers $k,d$ satisfying $k\leq d$, there exist a positive integer $c=c(d,k)$ and real number $C=C(d,k)>0$ such that the following holds. For every positive integer $n$ there exists a set $S\subset [n]^{d}$ of size at least $Cn^{d(d-k)/(d-1)}$ such that no $k$-dimensional linear hyperplane contains more than $c$ elements of $S$.
	\end{theorem}

	\begin{corollary}\label{cor:covering}
		Let $k,d$ be positive integers satisfying $k< d$, then there exists $C>0$ such that the following holds for every positive integer $n$. The number of $k$-dimensional hyperplanes in any covering of $[n]^{d}$ is at least $Cn^{d(d-k)/(d-1)}$.
	\end{corollary}
	
	We will give a very short alternative proof of Corollary \ref{cor:covering} as well, which does not rely on Theorem~\ref{thm:linear}. Finally, let us remark that $c=c(d,k)$ denotes the same function in Theorems \ref{thm:main}, \ref{thm:affine} and \ref{thm:linear}.
	
	\subsection{Point-hyperplane incidences}
	
	One of the fundamental results in combinatorial geometry is the Szemer\'edi-Trotter theorem \cite{SzT}, which states that the number of incidences between $n$ points and $m$ lines is $O((mn)^{2/3}+m+n)$, and this bound is the best possible. Extending this result to higher dimensions is a notorious open problem. Given a set of points $P$ and set of hyperplanes $\mathcal{H}$ in $\mathbb{R}^d$, let $I(P,\mathcal{H})$ denote the number of incidences between $P$ and $\mathcal{H}$, that is, the number of pairs $(p,H)\in P\times \mathcal{H}$ such that $p\in H$. Note that in $\mathbb{R}^{3}$, by taking $n$ points on a single line and $m$ planes containing this line, we have a collection of $n$ points and $m$ planes with $mn$ incidences. Therefore, in order to avoid this triviality, we forbid a complete bipartite graph $K_{c,c}$ in the incidence graph of the configuration. I.e., if $P$ is a set of $n$ points and $\mathcal{H}$ is a set of $m$ hyperplanes in $\mathbb{R}^d$, we are interested in the maximum of $I(P,\mathcal{H})$ as a function of $m$ and $n$ assuming there are no $c$ hyperplanes containing the same $c$ points. Let $f(d,n,m,c)$ denote this maximum.
	
	It follows from works of Chazelle \cite{Ch}, Brass and Knauer \cite{BK03} and Apfelbaum and Sharir \cite{AS} that $$f(d,n,m,c)=O_{d,c}((mn)^{1-1/(d+1)}+m+n).$$
	However, this bound is only known to be sharp in case $d=2$.  Brass and Knauer \cite{BK03} observed that large sets of lattice points satisfying the conditions of Theorems \ref{thm:affine} and \ref{thm:linear} can be used to provide lower bounds for $f(d,n,m,c)$. For every pair of integers $m$ and $n$, and real number $\varepsilon>0$, they showed that there exists $c$ such that 
	$$f(d,n,m,c)\geq \begin{cases} 
		(mn)^{1-2/(d+3)-\varepsilon} &\mbox{ if }d\mbox{ is odd and }d>3,\\
		(mn)^{1-2(d+1)/(d+2)^2-\varepsilon} &\mbox{ if }d\mbox{ is even},\\
		\Omega((mn)^{7/10})&\mbox{ if }d=3.
	\end{cases}$$
	By improving the known lower bounds on $\ell(d,k,n,c)$, Balko, Cibulka, and Valtr \cite{BCV} improved the lower bounds on $f(d,n,m,c)$ as well for $d\geq 4$. By using Theorems \ref{thm:affine} and \ref{thm:linear}, we further improve their result, and as these theorems are optimal (up to the value of $c$), we reach the full potential of the approach outlined by Brass and Knauer \cite{BK03}. 
	
	\begin{theorem}\label{thm:incidence}
		For every positive integer $d$ there exists $c$ such that the following holds. Let $m,n$ be positive integers, then there exists a set of $n$ points $P$ and a set of $m$ hyperplanes $\mathcal{H}$ in $\mathbb{R}^{d}$ such that the incidence graph of $P$ and $\mathcal{H}$ is $K_{c,c}$-free, and 
		$$I(P,\mathcal{H})\geq \begin{cases} 
			\Omega_{d}((mn)^{1-(2d+3)/(d+2)(d+3)}) &\mbox{ if }d\mbox{ is odd,}\\
			\Omega_{d}((mn)^{1-(2d^2+d-2)/(d+2)(d^2+2d-2)}) &\mbox{ if }d\mbox{ is even}.
		\end{cases}$$
	\end{theorem}
	
	\noindent	
	In certain asymmetric settings, i.e when $n$ is much larger than $m$, better bounds are known, see \cite{S16}.
	
	\medskip
	
	The rest of this paper is organized as follows. In Section \ref{sect:lowerbound}, we prove Theorems \ref{thm:lowerbound} and \ref{thm:affine}. Then, in Section \ref{sect:vectors}, we prove Theorem \ref{thm:main}. In Section \ref{sect:covering}, we prove Theorem \ref{thm:linear}, and give an alternative proof of Corollary \ref{cor:covering}. Finally, in Section \ref{sect:incidence}, we give a proof sketch of Theorem \ref{thm:incidence}.
	
	\section{Lower bounds for evasiveness}\label{sect:lowerbound}
	
	In this section, we prove Theorems \ref{thm:lowerbound} and \ref{thm:affine}. In order to prove Theorem \ref{thm:lowerbound}, we consider a variant of the Erd\H{o}s Box theorem \cite{E64}. This theorem is a generalization of the K\H{o}v\'ari-S\'os-Tur\'an theorem \cite{KST}, providing upper bounds on the maximum number of edges of an $r$-partite $r$-uniform hypergraph with parts of size $n$ containing no copy of the complete $r$-partite $r$-uniform hypergraph $K_{s_1,\dots,s_r}$. As we require a version of the Box theorem in which the parts of the host hypergraph have different sizes (which is not a standard setting), we present a short proof of the result that we need. With slight abuse of notation, given an $r$-uniform $r$-partite hypergraph $H$ with vertex classes $V_1,\dots,V_r$, we view edges of $H$ as both $r$-element subsets of the vertex set, and elements of the Cartesian product $V_1\times\dots\times V_r$. We also denote by $X^{(s)}$ all $s$-element subsets of the set $X$.
	
	\begin{lemma}\label{lemma:boxthm}
		Let $r$ and $s_1,\dots,s_r\geq 2$ be positive integers. Let $H$ be an $r$-partite $r$-uniform hypergraph with vertex classes $V_1,\dots,V_r$ such that $|V_i|\geq s_i^2|V_r|^{\frac{1}{s_i\dots s_{r-1}}}$ for $i\in [r-1]$. If $H$ has at least $$2s_r^{\frac{1}{s_1\dots s_{r-1}}}|V_1|\dots |V_{r-1}||V_r|^{1-\frac{1}{s_1\dots s_{r-1}}}$$ edges, then there exists $S_1\subset V_1,\dots,S_r\subset V_r$ such that $|S_i|=s_i$ for $i\in [r]$, and $S_1\times \dots\times S_r\subset E(H)$.
	\end{lemma}
	
	\begin{proof}
	We prove this by induction on $r$. In case $r=1$, $H$ has at least $2s_1$ edges, so the statement is true. Let us assume that $r\geq 2$. Let $U=V_2\times\dots\times V_r$ and let 
	$$t\geq 2s_r^{\frac{1}{s_1\dots s_{r-1}}}|V_1|\dots |V_{r-1}||V_r|^{1-\frac{1}{s_1\dots s_{r-1}}}$$
	be the number of edges of $H$. For each $f\in U$, let $d(f)$ denote the number of edges of $H$ containing $f$. Also, for every set of vertices $W\subset V_1$, let 
		$$N(W)=\{f\in U:\forall v\in W,\{v\}\cup f\in E(H)\}.$$ Then we have the following equality:
		$$\sum_{W\in V_1^{(s_1)}}|N(W)|=\sum_{f\in U}\binom{d(f)}{s_1}.$$
		By the convexity of the function $\binom{x}{s_1}$, and recalling that $\sum_{f\in U}d(f)=t$, we can write the following inequality:
		$$\sum_{f\in U}\binom{d(f)}{s_1}\geq |U|\binom{t/|U|}{s_1}\geq \frac{t^{s_1}}{2s_1!|U|^{s_1-1}}.$$
		The last inequality holds by the condition $t/|U|\geq 2|V_1||V_r|^{-\frac{1}{s_1\dots s_{r-1}}}>2s_1^2$.
	Therefore, by the pigeonhole principle, there exists $S_1\in V_1^{(s_1)}$ such that 
		$$|N(S_1)|\geq \frac{t^{s_1}}{2s_1!|U|^{s_1-1}\binom{|V_1|}{s_1}}\geq \frac{t^{s_1}}{2|V_1|^{s_1}|U|^{s_1-1}}
\geq 2s_r^{\frac{1}{s_2\dots s_{r-1}}}|V_2|\dots |V_{r-1}||V_r|^{1-\frac{1}{s_2\dots s_{r-1}}}$$
	Let $H'$ be the $(r-1)$-partite $(r-1)$-uniform hypergraph with vertex classes $V_2,\dots,V_{r}$ and set of edges $E(H')=N(S_1)$. Then we can apply our induction hypothesis to conclude that there exist $S_2\subset V_2,\dots,S_r\subset V_r$ such that $|S_i|=s_i$ for $i=2,\dots,r$, and $S_2\times\dots\times S_r\subset E(H')$. But then $S_1\times \dots\times S_r\subset V_1\times\dots\times V_r$, so $S_1,\dots,S_r$ satisfy the required properties. 

		
		
%
	\end{proof}
	
	Now we are ready to prove Theorem \ref{thm:lowerbound}.
	
	\begin{proof}[Proof of Theorem \ref{thm:lowerbound}]
		Let us introduce some parameters. Let $r=\lfloor \log_2 (1/\varepsilon)\rfloor-1$, then we may assume that $r\leq k$, otherwise the statement of the theorem is vacuous. For $i\in [r-1]$, let $t_i=\lceil \frac{d2^{i+1-r}}{3}\rceil$, and set $T=t_1+\dots+t_{r-1}$. Observe that $T<\frac{2d}{3}$,  assuming $d$ is sufficiently large with respect to $r$. Furthermore,  for $i=1,\dots,r-1$, let $V_i=\mathbb{F}^{t_i}$, and let $V_{r}=\mathbb{F}^{d-T}$. We will view $\mathbb{F}^d$ as the Cartesian product $V_1\times\dots\times V_{r}$. Define the $r$-partite $r$-uniform hypergraph $H$ on the vertex classes $V_1,\dots,V_r$ such that $v\in V_1\times\dots\times V_r$ is an edge if $v\in S$.
		
		We would like to apply Lemma \ref{lemma:boxthm} with  $s_1=\dots=s_{r-1}=2$ and $s_r=k-r+2$ to the hypergraph $H$ to find suitable sets $S_1,\dots,S_r$. However, in order to do this, we need to verify that $H$ satisfies the conditions of the lemma. First of all, for $i\in [r-1]$, we have 
		$$|V_i|=|\mathbb{F}|^{t_i}\geq |\mathbb{F}|^{\frac{d2^{i+1-r}}{3}}\geq 4|\mathbb{F}|^{(d-T)2^{i-r}}=s_i^{2}|V_r|^{\frac{1}{s_i\dots s_{r-1}}},$$
		where the second inequality holds assuming $d$ is sufficiently large with respect to $r$. Furthermore, note that $\frac{1}{8\varepsilon}<s_1\dots s_{r-1}=2^{r-1}\leq \frac{1}{4\varepsilon}$, and $$\frac{d-T}{s_1\dots s_{r-1}}\geq 4\varepsilon(d-T)>\frac{4\varepsilon d}{3}.$$ Therefore, we can write
		$$2s_r^{\frac{1}{s_1\dots s_{r-1}}}|V_1|\dots |V_{r-1}||V_r|^{1-\frac{1}{s_1\dots s_{r-1}}}<2k^{8\varepsilon}|\mathbb{F}|^{d(1-4\varepsilon/3)}\leq |S|.$$ 
		Here, the last inequality holds by assuming $d$ is sufficiently large with respect to $k$. Thus, the conditions of Lemma \ref{lemma:boxthm} are satisfied, so we can find $S_1\subset V_1,\dots,S_{r}\subset V_r$ such that $|S_i|=s_i$ for $i\in [r]$, and $W=S_1\times \dots\times S_r\subset S$. Let $S_i=\{u_i,v_i\}$ for $i\in [r-1]$, and let $S_r=\{w_0\dots w_{k-r+1}\}$. Given $w\in V_i$ for some $i\in [r]$, let $w'\in \mathbb{F}^d$ denote the vector which agrees with $w$ on $V_i$, and vanishes on all other coordinates. Then $W$ is contained in the affine subspace 
		$$\left(w_0'+\sum_{i=1}^{r-1}u_1'\right)+\mbox{span}\langle\{v_i'-u_i': i\in [r-1]\}\cup \{w_i'-w_0':i\in [k-r+1]\}\rangle,$$
		which clearly has dimension at most $k$. Finally, as $|W|=2^{r-1}(k-r+1)>\frac{k-\log_2(1/\varepsilon)}{8\varepsilon}$, this shows that $S$ is not $(k,\frac{k-\log_2(1/\varepsilon)}{8\varepsilon})$-subspace evasive.
	\end{proof}
	
	Finally, let us present the proof of Theorem \ref{thm:Hammingbound}.
	
	\begin{proof}[Proof of Theorem \ref{thm:Hammingbound}]
		For the convenience of the reader, we  first  recall the proof of the Hamming bound, that is (\ref{eq:Hamming}). Let $S\subset \mathbb{F}^{d}$ be a set of vectors such that no $k$-dimensional linear subspace contains $k+1$ elements. Without loss of generality, assume that $S$ spans $\mathbb{F}^{d}$. Let $M\in\mathbb{F}^{d\times |S|}$ be a matrix, whose columns are the elements of $S$. Then $L=\mbox{ker}(M)<\mathbb{F}^{|S|}$ does not contain a vector with at most $k+1$ non-zero coordinates, which implies that $L$ is an $[|S|,|S|-d,k+2]$-code. Hence, the Hamming balls of radius $r=\lfloor \frac{k+1}{2}\rfloor$ around the elements of $L$ are disjoint. The size of such a ball is at least $(|\mathbb{F}|-1)^r\binom{|S|}{r}\geq 2^{-r}|\mathbb{F}|^{r}\binom{|S|}{r}$, which gives that $|L|\cdot 2^{-r}|\mathbb{F}|^r\binom{|S|}{r}\leq |\mathbb{F}|^{|S|}$. From this, we get $\binom{|S|}{r}\leq 2^{r}|\mathbb{F}|^{d-r}$, which further implies $|S|\leq 2k|\mathbb{F}|^{d/r-1}$.
		
		\medskip
		
		Now let us turn to the proof of Theorem \ref{thm:Hammingbound}. Write $k=k_1+\dots+k_{C+1}$, where $k_i\in \{\lfloor \frac{k}{C+1}\rfloor,\lceil \frac{k}{C+1}\rceil\}$ for $i\in [C+1]$. Note that
		$\lfloor \frac{k_i+1}{2}\rfloor>\lfloor \frac{k}{2(C+1)}\rfloor.$
		Hence, by the above discussion, if $|S|\geq 2k|\mathbb{F}|^{\frac{d}{\lfloor k/2(C+1)\rfloor}}+k+C+1$, we can find  disjoint subsets $W_1,\dots,W_{C+1}\subset S$ such that $|W_{i}|=k_i+1$ and $W_{i}$ spans a linear subspace of dimension at most $k_i$ for $i\in [C+1]$. Indeed, select $W_1,\dots,W_{C+1}$ one-by-one, at each step deleting the selected set from $S$. But then $W=W_1\cup\dots \cup W_{C+1}$ spans a linear subspace of dimension at most $k$, and $|W|=k+C+1$, showing that $S$ is not $(k,k+C)$-evasive.
	\end{proof}
	
	\section{Optimal constructions of evasive sets}\label{sect:vectors}
	
	In this section, we give an alternative proof of Theorem \ref{thm:main}. Our proof is based on the random algebraic method pioneered by Bukh, and uses the ideas from his paper \cite{Bukh}. With slight abuse of notation, let us exchange $k$ with $d-k$ for our (and the reader's) future convenience, so we prove the following equivalent formulation of Theorem \ref{thm:main}.
	
	\begin{theorem}\label{thm:main2}
		For every pair of positive integers $k,d$ satisfying $k\leq d$, there exists a positive integer $c=c(d,d-k)$ such that the following holds. For every finite field $\mathbb{F}$, there exists a $(d-k,c)$-subspace evasive set of size $|\mathbb{F}|^{k}$ in $\mathbb{F}^{d}$.
	\end{theorem}
	
	Let $D=(d+1)k+1$, $p=|\mathbb{F}|$, and let $\mathcal{Q}_{D}<\mathbb{F}[x_1,\dots,x_k]$ denote the space of polynomials of (total) degree at most $D$ on $k$ variables. Write $$\Lambda_{D}=\{\alpha\in\mathbb{N}^{k}:\alpha(1)+\dots+\alpha(k)\leq D\},$$ which is the set of possible exponents of the monomials of the polynomials in $\mathcal{Q}_{D}$. Let $q_1,\dots,q_{d}$ be random elements of $\mathcal{Q}_{D}$ chosen independently from the uniform distribution, and set $\mathbf{q}=(q_1,\dots,q_d)$. Our goal is to show that the set 
	$$S=\{\mathbf{q}(\mathbf{x}):\mathbf{x}\in\mathbb{F}_p^k\}$$
	has the property that the no $(d-k)$-dimensional affine subspace of $\mathbb{F}^{d}$ contains more than $c$ elements of $H$ with high probability, if $c$ is sufficiently large with respect to $k$ and $d$.
	
	We prepare the proof of this with a number of claims. First, let us state three simple observations that we will use repeatedly.
	\begin{itemize}
		\item[(i)] 	If $M\in\mathbb{F}^{k\times d}$ has rank $k$, and $\mathbf{v}\in\mathbb{F}^{d}$ is chosen randomly from the uniform distribution, then $M\mathbf{v}$ is uniformly distributed in $\mathbb{F}^{k}$.
		
		\item[(ii)]  If $X_1,\dots,X_d$ are uniformly distributed random variables in $\mathbb{F}$, then $X_1,\dots,X_d$ are independent if and only if $(X_1,\dots,X_d)$ is uniformly distributed in $\mathbb{F}^{d}$.
		
		\item[(iii)] If $X_1,\dots,X_d$ are independent, uniformly distributed random variables on $\mathbb{F}$, and $Y_1,\dots,Y_d$ are random variables on $\mathbb{F}$ such that $X_i$ and $Y_j$ are independent for any $i, j\in [d]$, then $X_1+Y_1,\dots,X_d+Y_d$ are independent and uniformly distributed. 
	\end{itemize}	
	
	\begin{claim}\label{claim:basis}
		Let $\mathbf{v}_1,\dots,\mathbf{v}_k\in \mathbb{F}^{d}$ be linearly independent vectors. Then the polynomials $\langle \mathbf{q},\mathbf{v}_1\rangle,\dots,\langle\mathbf{q},\mathbf{v}_k\rangle$ are independent and uniformly distributed in $\mathcal{Q}_{D}$.
	\end{claim}
	
	\begin{proof}
		Let $M\in\mathbb{F}^{d\times k}$ be the matrix, whose rows are $\mathbf{v}_1,\dots,\mathbf{v}_k$. For $\alpha\in \Lambda_{D}$ and $i\in [d]$, let $c_{i,\alpha}$ be the coefficient of the monomial $\mathbf{x}^{\alpha}=x_{1}^{\alpha(1)}\dots x_{k}^{\alpha(k)}$ in $q_{i}$, and let $\mathbf{c}_{\alpha}=(c_{1,\alpha},\dots,c_{d,\alpha})$. Observe that the $d\cdot|\Lambda_{D}|$ random variables $(c_{i,\alpha})_{i\in [d], \alpha\in\Lambda_{D}}$ are independent and uniformly distributed in~$\mathbb{F}$. 
		
		The coefficient of $\mathbf{x}^{\alpha}$ in $\langle \mathbf{q},\mathbf{v}_i\rangle$ is $(M\mathbf{c}_{\alpha})(i)$. As $M\mathbf{c}_{\alpha}$ is uniformly distributed in $\mathbb{F}^{k}$ and $(\mathbf{c}_{\alpha})_{\alpha\in \Lambda_{D}}$ are independent, this proves the claim.
	\end{proof}
	
	\begin{claim}\label{claim:uniform}
		Let $\mathbf{z}\in\mathbb{F}^{k}$ and $i\in [d]$. Then $q_{i}(\mathbf{z})$ is uniformly distributed in $\mathbb{F}$.
	\end{claim}
	
	\begin{proof}
		This follows as the constant term of $q_i$ is uniformly distributed in $\mathbb{F}$.
	\end{proof}
	
	\begin{claim}\label{claim:independent}
		Let $s\leq \min\{D,|\mathbb{F}|^{1/2}\}$,  and let $\mathbf{z}_1,\dots,\mathbf{z}_s\in \mathbb{F}^{k}$ be pairwise distinct vectors. Then the $d\cdot s$ random variables $(q_{i}(\mathbf{z}_j))_{i\in [d],j\in [s]}$ are independent.
	\end{claim}
	
	\begin{proof}
		First, suppose that the first coordinates of the vectors $\mathbf{z}_{1},\dots,\mathbf{z}_{s}$ are pairwise distinct. For $\alpha\in \{0,1,\dots,s-1\}$ and $i\in[d]$, let $c_{i,\alpha}$ be the coefficient of $x_{1}^{\alpha}$ in $q_i$. Also, let $\mathbf{c}_i=(c_{i,0},\dots,c_{i,s-1})$ and $\mathbf{y}_i=(1,\mathbf{z}_i(1),\mathbf{z}_i(1)^{2},\dots,\mathbf{z}_i(1)^{s-1})$. Then $\mathbf{y}_1,\dots,\mathbf{y}_s$ are linearly independent, using that $\mathbf{z}_1(1),\dots,\mathbf{z}_{s}(1)$ are pairwise distinct, and so the Vandermonde determinant is nonzero. Let $M\in\mathbb{F}_p^{s\times s}$ be the matrix whose rows are $\mathbf{y}_1,\dots,\mathbf{y}_s$. Then $M$ has rank $s$, so $M\mathbf{c}_i$ is uniformly distributed in $\mathbb{F}_p^{s}$. As $\mathbf{c}_1,\dots,\mathbf{c}_d$ are independent, we get that the $d\cdot s$ numbers $((M\mathbf{c}_i)(j))_{i\in[d],j\in [s]}$ are independent. But $q_{i}(\mathbf{z}_j)=X_{i,j}+Y_{i,j}$, where $X_{i,j}=(M\mathbf{c}_i)(j)$, and $X_{i,j}$ and $Y_{i',j'}$ are independent (since these variables depend on disjoint sets of random coefficients), hence $(q_{i}(\mathbf{z}_j))_{i\in[d],j\in [s]}$ are independent as well (see (iii)).
		
		Now consider the general case. We show that there exists an invertible matrix $M\in\mathbb{F}^{k\times k}$ such that $M\mathbf{z}_1,\dots,M\mathbf{z}_s$ have pairwise distinct first coordinates. As $M$ is a change of basis, the polynomial $q_i'$ defined as $q_i'(\mathbf{x})=q_i(M^{-1}\mathbf{x})$ is also uniformly distributed in $\mathcal{Q}_{p,D}$, so then we are done by the previous argument. Choose $M$ randomly from the uniform distribution on all invertible matrices. Then for $1\leq i<j\leq s$, we have $\mathbb{P}(M\mathbf{z}_i(1)=M\mathbf{z}_j(1))=(|\mathbb{F}|^{k-1}-1)/(|\mathbb{F}|^{k}-1)<1/|\mathbb{F}|$ as $M(\mathbf{z}_i-\mathbf{z}_j)$ is uniformly distributed on $\mathbb{F}^{k}\setminus\{0\}$. Hence, by Markov's inequality, the probability that there exists $1\leq i<j\leq s$ such that $M\mathbf{z}_i(1)=M\mathbf{z}_j(1)$ is at most $\binom{s}{2}/|\mathbb{F}|<1$, implying the existence of the desired matrix $M$.
	\end{proof}
	
	Let $V$ be a $(d-k)$-dimensional affine subspace of $\mathbb{F}^{d}$ and let $\mathbf{z}\in\mathbb{F}^{k}$. Let $I(\mathbf{z},V)$ be the indicator random variable of the event $\{q(\mathbf{z})\in V\}$. Then there exist $k$ linearly independent vectors $\mathbf{v}_1,\dots,\mathbf{v}_k\in\mathbb{F}^{d}$ and $\mathbf{b}\in\mathbb{F}^{k}$ such that   $I(\mathbf{z},V)=1$ if and only if $\langle q(\mathbf{z}),\mathbf{v}_i\rangle=\mathbf{b}(i)$ for every $i\in [k]$.
	Furthermore, set $$N(V)=\sum_{\mathbf{z}\in \mathbb{F}^{k}} I(\mathbf{z},V).$$ Let $\mathbf{e}_1,\dots,\mathbf{e}_d\in\mathbb{F}^{d}$ be the unit basis, that is, $\mathbf{e}_i(j)=1$ if $i=j$, and $\mathbf{e}_i(j)=0$ otherwise. Let $E$ be the $(d-k)$-dimensional linear subspace with normal vectors $e_1,\dots,e_k$. By Claim \ref{claim:basis} and \ref{claim:uniform}, $N(V)$ has the same distribution as $N(E)$, so for simplicity, write $N=N(E)$ and $I(\mathbf{z})=I(\mathbf{z},E)$. Also, observe that $I(\mathbf{z})$ is the indicator random variable of the event $q_1(\mathbf{z})=\dots=q_k(\mathbf{z})=0$. Therefore, by Claim \ref{claim:uniform} and \ref{claim:independent}, we have that $\mathbb{P}(I(\mathbf{z})=1)=1/|\mathbb{F}|^{k}$, and if $s\leq \min\{D,|\mathbb{F}|^{1/2}\}$ and $\mathbf{z}_1,\dots,\mathbf{z}_s\in\mathbb{F}^{k}$ are distinct, then $I(\mathbf{z}_1),\dots,I(\mathbf{z}_s)$ are independent. 
	
	\begin{claim}\label{claim:expectation}
		Let $s\leq \min\{D,|\mathbb{F}|^{1/2}\}$. Then $\mathbb{E}(N^{s})\leq s^{s+1}$
	\end{claim}
	
	\begin{proof}
		We can write 
		$$\mathbb{E}(N^{s})=\sum_{\mathbf{z}_1,\dots,\mathbf{z}_s\in\mathbb{F}_p^{k}}\mathbb{E}(I(\mathbf{z}_1)\dots I(\mathbf{z}_s)).$$
		Here, the $s$-wise independence of the variables $I(\mathbf{z})$ guarantees that $\mathbb{E}(I(\mathbf{z}_1)\dots I(\mathbf{z}_s))=|\mathbb{F}|^{-kr}$, where $r$ is the number of different elements among $\mathbf{z}_1,\dots,\mathbf{z}_s$. The number of choices of $(\mathbf{z}_1,\dots,\mathbf{z}_s)$ containing $r$ distinct entries is at most $r^{s}p^{kr}$ (as there are at most $|\mathbb{F}|^{kr}$ choices for the $r$ vectors, and each $r$-tuple of vectors yields at most $r^{s}$  such $s$-tuples). Hence, we arrive to the bound
		$$\mathbb{E}(N^{s})\leq \sum_{r=1}^{s}r^{s}|\mathbb{F}|^{kr}\cdot |\mathbb{F}|^{-kr}<s^{s+1}.$$
	\end{proof}
	
	A crucial ingredient in the proof is the following fact from algebraic geometry, which says that a variety in $\mathbb{F}_p^{k}$ contains either at most a constant number of points (depending only on the degree of the variety), or at least $\Omega(p)$ points. In our case, this means that $N$ is either bounded by a constant, or at least $\Omega(p)$. But as the higher moments of $N$ are bounded by a constant, it is very unlikely that $N=\Omega(p)$.
	
	\begin{lemma}\label{lemma:variety}\cite{Bukh}
		For every $k$ and $D$ there exists a constant $c$ such that the following holds. Suppose that $q_1,\dots,q_k\in \mathbb{F}[x_1,\dots,x_k]$ are polynomials of degree at most $D$. Then the size of the variety 
		$$W=\{\mathbf{x}\in\mathbb{F}^{k}: q_1(\mathbf{x})=\dots=q_k(\mathbf{x})=0\}$$
		is either at most $c$, or at least $|\mathbb{F}|-c|\mathbb{F}|^{1/2}$.
	\end{lemma}	
	
	Now everything is prepared to prove our main theorem.
	
	\begin{proof}[Proof of Theorem \ref{thm:main2}] Clearly, it is enough to prove the theorem in case $|\mathbb{F}|$ is sufficiently large with respect to $d$ and $k$. Let $c$ be the constant given by Lemma \ref{lemma:variety} (with respect to $k$ and $D$), and  suppose that $|\mathbb{F}|>\max\{(2c)^2,(2D)^{D+1}\}$. 
		
		We show that $S=\{\mathbf{q}(\mathbf{z}):\mathbf{z}\in\mathbb{F}^{k}\}$ satisfies the assertion of the theorem with positive probability. Let $V$ be a $(d-k)$-dimensional affine subspace of $\mathbb{F}^{d}$. Then $|V\cap S|=N(V)$, so by Claim \ref{claim:expectation} applied with $s=D$, we have $\mathbb{E}(|V\cap S|^{D})\leq D^{D+1}$. Applying Markov's inequality, for every $\lambda>0$, we have
		$$\mathbb{P}(|V\cap S|\geq \lambda)\leq \mathbb{P}(|V\cap S|^D\geq \lambda^{D})\leq \frac{D^{D+1}}{\lambda^{D}}.$$
		But note that by Lemma \ref{lemma:variety}, we have either $|V\cap S|\leq c$, or $|V\cap S|\geq |\mathbb{F}|-c|\mathbb{F}|^{1/2}>|\mathbb{F}|/2$. Hence, we can further write
		$$\mathbb{P}(|V\cap S|>c)=\mathbb{P}\left(|V\cap S|\geq \frac{|\mathbb{F}|}{2}\right)\leq \frac{(2D)^{D+1}}{|\mathbb{F}|^D}.$$
		The number of different $(d-k)$-dimensional affine subspaces in $\mathbb{F}^{d}$ is at most $ (|\mathbb{F}|^{d})^{k}\cdot |\mathbb{F}|^{k}=|\mathbb{F}|^{(d+1)k}$, as there are at most $(|\mathbb{F}|^{d})^{k}$ choices for the $k$ normal vectors, and at most $|\mathbb{F}|^{k}$ translations. Therefore, the expected number of hyperplanes $V$ violating $|V\cap S|\leq c$ is at most $$\frac{|\mathbb{F}|^{(d+1)k}\cdot (2D)^{D+1}}{|\mathbb{F}|^D}\leq \frac{(2D)^{D+1}}{|\mathbb{F}|}<1,$$ recalling that $D=(d+1)k+1$. This finishes the proof.
	\end{proof}
	
	\section{Covering by hyperplanes}\label{sect:covering}
	In this section, we prove Theorem \ref{thm:linear}, and provide and alternative proof of Corollary \ref{cor:covering}.

	Given a prime $p$ and vectors $\mathbf{x},\mathbf{y}\in\mathbb{F}_p^{d}$, write $\mathbf{x}\sim \mathbf{y}$ if there exists $\lambda\in\mathbb{F}_p\setminus\{0\}$ such that $\mathbf{x}=\lambda \mathbf{y}$. A crucial observation is that the $\sim$ equivalence class of each vector contains an element, whose every coordinate is contained in the interval $[- p^{(d-1)/d}, p^{(d-1)/d}]$. This follows from Dirichlet's theorem on simultaneous approximations (see e.g. \cite{approx}, Chapter 2, Theorem 1A), but we also provide a simple proof for completeness.
	
	\begin{lemma}\label{lemma:approx}
		Let $d,n$ be positive integers, let $p\leq n^{d/(d-1)}$ be a prime, and let $\mathbf{x}\in\mathbb{F}_p^{d}\setminus\{0\}$. Then there exists $\mathbf{y}\in\mathbb{F}_p^d$ such that $\mathbf{x}\sim\mathbf{y}$ and $\mathbf{y}(i)\in [-n,n]$ for $i\in [d]$.
	\end{lemma}
	
	\begin{proof}
		For every $\mathbf{z}\in\mathbb{F}_p^{d}$ and positive integer $t$, define the ``ball of radius $t$ centered at $\mathbf{z}$'' as
		$$B_t(\mathbf{z})=\{\mathbf{v}\in\mathbb{F}_p^{d}: \forall i\in [d], \mathbf{v}(i)-\mathbf{z}(i)\in [-t,t]\}.$$
		Then for $t\leq (p-1)/2$, we have $|B_{t}(\mathbf{z})|=(2t+1)^{d}$. Let $t=p^{(d-1)/d}/2$, then $t$ satisfies $p\cdot(2t+1)^{d}>p^{d}$. Hence, by the pigeonhole principle, there exists $\lambda_1\neq \lambda_2\in\mathbb{F}_p$ such that $B_t(\lambda_1\mathbf{x})\cap B_{t}(\lambda_2\mathbf{x})\neq \emptyset$. But then setting $\mathbf{y}=(\lambda_1-\lambda_2)\mathbf{x}\neq 0$, we have $\mathbf{x}\sim\mathbf{y}$ and $\mathbf{y}(i)\in [-2t,2t]\subset [-n,n]$.
	\end{proof}
	
	\begin{proof}[Proof of Theorem \ref{thm:linear}]
		Let $p$ be a prime such that $n^{d/(d-1)}/2<p<n^{d/(d-1)}$, which exists by Bertrand's postulate. Let $c=c(d,k)$ be the constant provided by Theorem \ref{thm:main}, and let $S_0\subset \mathbb{F}_p^{d}$ be a set of $p^{d-k}$ vectors such that no $k$-dimensional (linear) subspace contains more than $c$ elements of $S_0$. By Lemma \ref{lemma:approx}, for every $\mathbf{x}\in S_0$ there exists $\mathbf{x}^{*}\in [-n,n]^{d}$ such that $\mathbf{x}^{*}\equiv\lambda \mathbf{x} \pmod{p}$ for some $\lambda\neq 0$. In particular, no $k$-dimensional linear hyperplane contains more than $c$ elements of $S^{*}=\{\mathbf{x}^{*}:\mathbf{x}\in S_0\}$. By the pigeonhole principle, there exists $S'\subset S^{*}$ of size at least $p^{d-k}/3^d$ such that every element of $S'$ have the same sign-pattern. Let $S$ be the set of vectors we get after changing the 0 entries of the elements of $S'$ to $1$, and multiplying the negative coordinates by -1. Then $S$ is contained in $[n]^{d}$, it has at least $p^{d-k}/3^{d}\geq n^{d(d-k)/(d-1)}/6^{d}$ elements, and it is easy to check that no $k$-dimensional linear hyperplane contains more than $c$ elements of $S$.
	\end{proof}
	
	\begin{proof}[Proof of Corollary \ref{cor:covering}]
		For slight convenience, we consider the grid $[-n,n]^{d}$ instead of $[n]^{d}$. This does not change the problem up to the value of $C$ for the following reason. If $[n]^{d}$ can be covered by $N$ linear hyperplanes of dimension $k$, then $[-n,n]^{d}$ can be covered by $3^{d}N$ linear hyperplanes of dimension $k$, as we can partition $[-n,n]^{d}$ into $3^{d}$ parts with respect to the signs of the vectors, and each part requires at most $N$ hyperplanes. 
		
		Let $p$ be a prime such that $n^{d/(d-1)}/2<p<n^{d/(d-1)}$, which exists by Bertrand's postulate. For every $\mathbf{x}\in\mathbb{F}_p^{d}$, let $\mathbf{x}^{*}\in [-n,n]^{d}$ be an arbitrary vector such that $\mathbf{x}\sim \mathbf{x}^{*}$, and let $$S=\{\mathbf{x}^{*}:\mathbf{x}\in\mathbb{F}_p^{d}\setminus\{0\}\}\subset [-n,n]^{d}.$$
		Then $|S|=(p^{d}-1)/(p-1)\geq p^{d-1}$.
		
		Suppose that $S'\subset S$ spans a linear hyperplane of dimension at most $k$ over $\mathbb{R}$. Then $S'$ spans a subspace of $\mathbb{F}_p^{d}$ of dimension at most $k$. As $S'$ contains at most one element of each equivalence class of $\sim$, we get that $|S'|\leq (p^{k}-1)/(p-1)\leq 2p^{k-1}.$ Hence, any covering of $S$ with linear hyperplanes contains at least $|S|/2p^{k-1}\geq p^{d-k}/2\geq n^{d(d-k)/(d-1)}/2^{k+1}$ elements.
	\end{proof}
	
	\section{Incidences}\label{sect:incidence}
	
	As the proof of Theorem \ref{thm:incidence} is essentially identical to the proofs of \cite{BK03} and \cite{BCV}, let us only give a very brief outline of it.
	
	\begin{proof}[Proof sketch of Theorem \ref{thm:incidence}]
	Let $k=\lfloor d/2\rfloor-1$,  $n_0\approx n^{1/(d-k)}$ and $m_0\approx (m/n_0)^{(d-1)/(dk+2d-1)}$. Let $P\subset [n_0]^{d}$ be a maximal set of lattice points such that no  $k$-dimensional affine subspace contains more than  $c_1=c(d,k)$ points of $P$, then $|P|\approx n_0^{d-k}\approx n$ by Theorem \ref{thm:affine}. Also, let $N\subset [m_0]^{d}$ be a maximal set of lattice points such that no $(d-k-1)$-dimensional linear subspace contains more than  $c_2=c(d,d-k-1)$ points of $N$, then $|N|\approx m_0^{d(k+1)/(d-1)}$ by Theorem \ref{thm:linear}. Let $\mathcal{H}$ be the set of all hyperplanes whose normal vector is in $N$ and  contains at least one point of $P$. Then $|\mathcal{H}|\lessapprox m_0n_0 |N|\approx m$ as the scalar product $\langle \mathbf{x},\mathbf{y}\rangle$ for any $\mathbf{x}\in P$ and $\mathbf{y}\in N$ is contained in $[dm_0n_0]$. Furthermore, the incidence graph of $(P,\mathcal{H})$ is $K_{c_1,c_2}$-free, as the intersection of any $c_2+1$ elements of $\mathcal{H}$ is an at most a $k$-dimensional affine hyperplane. Finally, $I(P,\mathcal{H})= |P||N|$, as for each $\mathbf{y}\in N$, the hyperplanes in $\mathcal{H}$ with normal vector $\mathbf{y}$ form a partition of $P$. Plugging in our bounds on $|P|$ and $|N|$ gives the desired result. See \cite{BCV} for the precise calculations, that give almost the same bounds.
	\end{proof}

	\vspace{0.3cm}
	\noindent	
	{\bf Acknowledgements.}
	We would like to thank Noga Alon and David Conlon for their valuable remarks and for pointing out the relevant references. Furthermore, we learned that D. Conlon (private communication) also obtained a proof of Theorem \ref{thm:main} using the random algebraic method.
	
	\vspace{0.15cm}
	\noindent
	Both authors were supported by the SNSF grant 200021\_196965.

\end{document}